\documentclass[11pt,a4paper,twoside]{article}
\usepackage{amsmath, amsthm, amscd, amsfonts, amssymb, graphicx, color}
\usepackage[bookmarksnumbered, colorlinks, plainpages]{hyperref}
\input{mathrsfs.sty}

\newtheorem{theorem}{Theorem}[section]

\newtheorem{lemma}[theorem]{Lemma}

\newtheorem{corollary}[theorem]{Corollary}
\theoremstyle{definition}

\newtheorem{example}[theorem]{Example}

\theoremstyle{remark}

\newtheorem{remark}[theorem]{Remark}
\begin{document}
\bigskip \title{\Large {\bf Linear combinations of univalent harmonic mappings convex in the direction of the imaginary axis}}
\author{{Raj Kumar\,\thanks{e-mail: rajgarg2012@yahoo.co.in},\, Sushma Gupta ~and~ Sukhjit Singh }}
\markboth{\small{Raj Kumar, Sushma Gupta ~and~ Sukhjit Singh }}{\small{Linear combination of univalent harmonic mappings}}
\date{}
\maketitle
{\footnotesize
{\bf Abstract.} In the present paper, we introduce a family of univalent harmonic functions, which map the unit disk onto domains convex in the direction of the imaginary axis. We find conditions for the linear combinations of mappings from this family to be univalent and convex in the direction of the imaginary axis. Linear combinations of functions from this family and harmonic mappings obtained by shearing of analytic vertical strip maps are also studied.

}

\section{Introduction}
 A complex-valued continuous function $f = u+iv$ is said to be harmonic in the open unit disk  $E = \{z: |z|<1\}$ if both $u$ and $v$ are real-valued harmonic functions in $E$. Such harmonic mappings have canonical decomposition $f=h+\overline g$, were $h$ is known as the analytic and $g$ the co-analytic part of $f$. A harmonic mapping $f=h+\overline g$ defined in $E$, is locally univalent and sense-preserving if and only if $h'\not=0$ in $E$ and the dilatation function $\omega,$ defined by $\displaystyle\omega=g'/h'$, satisfies $|\omega|<1$ in $E$. The class of all harmonic, univalent and sense-preserving mappings $f=h+\overline g$ in $E$ and normalized by the conditions $ f(0)=0 $ and $f_{z}(0)=1$ is denoted by $S_H$. Therefore, a function $f=h+\overline g$ in the class $S_H$ has the representation,
\begin{equation} f(z) = z+ \sum _{n=2}^{\infty} a_nz^n + \sum _{n=1}^{\infty}\overline{ b}{_n}\overline{z}^n, \end{equation} for all $z$ in $E$. The class of functions of the type (1) with $b_1=f_{\overline{z}}(0) = 0$ is a subclass of $S_H$ and will be denoted here by $S_H^0$. \\
\indent A domain $\Omega$ is said to be convex in a direction $\phi,\, 0 \leq \phi < \pi,$ if every line parallel to the line through $0$ and $ e ^{i \phi}$ has either connected or empty intersection with $\Omega$. The following result due to Hengartner and Schober [\ref{he and sc}] is very useful to check the convexity of an analytic function in the direction of the imaginary axis.
\begin{lemma} Suppose $f$ is analytic and non constant in $E$. Then $$\Re[(1-z^2)f'(z)]\geq0,\,\,z\in E$$ if and only if\\(i) $f$ is univalent in $E$;\\(ii) $f$ is convex in the direction of the imaginary axis;\\ (iii) there exist sequences $\{z_n'\}$ and $\{z_n''\}$ converging to $z=1$ and $z=-1$, respectively, such that $$lim_{n\rightarrow \infty} \,\,\Re(f(z_n'))= sup_{|z|<1}\,\,\Re(f(z)),$$ \begin{equation} \end{equation} $$ lim_{n\rightarrow \infty} \,\,\Re(f(z_n''))= inf_{|z|<1}\,\,\Re(f(z)). $$
\end{lemma}
 Construction of univalent harmonic mappings is not a very easy and straight forward task. In 1984, Clunie and Sheil-Small introduced a method, known as shear construction or shearing, for constructing a univalent harmonic mapping from a related conformal map. The following result of  Clunie and Sheil-Small [\ref{cl and sh}] is fundamental for constructing a harmonic univalent map convex in a given direction.
 \begin{lemma} A locally univalent harmonic function $f=h+\overline g$  in $E$ is a univalent harmonic mapping of $E$ onto a domain convex in a direction $\phi$ if and only if $h-e^{2i\phi}g$ is a univalent analytic mapping of $E$ onto a domain convex in the direction $\phi$.\end{lemma}
Another way of constructing desired univalent harmonic maps is by taking the linear combination of two suitable harmonic maps. For example in the following result Dorff [\ref{do 2}] identified two suitable harmonic functions $f_1$ and $f_2$ whose linear combination is univalent and convex in the direction of the imaginary axis.
\begin{theorem} Let $f_1=h_1+\overline{g}_{1}$ and $f_2=h_2+\overline{g}_{2}$ be two univalent harmonic mappings convex in the direction of the imaginary axis with $\omega_1=\omega_2$, where $\displaystyle \omega_1=g_1'/h_1'$ and $\displaystyle \omega_2=g_2'/h_2'$ are dilatation functions of $f_1$ and $f_2$ respectively. If $f_1$ and $f_2$ satisfy the conditions (2) above, then $f_3=t f_1+(1-t)f_2,\, 0\leq t\leq1$, is univalent and convex in the direction of the imaginary axis. \end{theorem}

In a recent paper, Wang et al.[\ref{wa and li}] derived several sufficient conditions on harmonic univalent functions $f_1$ and $f_2$ so that their linear combination $f_3=t f_1+(1-t)f_2,\,0\leq t\leq1,$ is univalent and convex  in the direction  of the real axis. In particular they established:
 \begin{theorem} Let $f_j=h_j+\overline{g}_{j} \in S_H$ with $\displaystyle h_j(z)+g_j(z)=z/(1-z)$ for $j=1,2$. Then $f_3=tf_1+(1-t)f_2,\, 0\leq t\leq 1$, is univalent and convex in the direction of the real axis.\end{theorem}
\indent From the above two papers it is observed that dilatation functions of $f_1$ and $f_2$ play an important role in deciding the behavior of their linear combinations. In the present paper, our aim is to study linear combinations of functions from the family of locally univalent and sense-preserving harmonic functions $f_{\alpha}=h_{\alpha}+\overline{g}_{\alpha},$ obtained by shearing of $\displaystyle F_\alpha(z)=h_{\alpha}(z)+g_{\alpha}(z)= \displaystyle {z(1-\alpha z)}/{(1-z^2)},\,\alpha\in[-1,1].$ Linear combinations of $f_{\alpha}$ and $f_{\theta}$ are also studied, where $f_{\theta}=h_{\theta}+\overline{g}_{\theta}$ is the harmonic function obtained by shearing of analytic vertical strip mapping \begin{equation}  h_\theta(z)+g_\theta(z)= \frac{1}{2i\,\sin\theta} \log\left(\frac{1+ze^{i\theta}}{1+ze^{-i\theta}}\right), \theta\in(0,\pi).\end{equation}
\section{Main Results}
 Let $$f_{\alpha}=h_{\alpha}+\overline{g}_{\alpha},\,{\rm  where}\,\, \displaystyle F_\alpha(z)=h_{\alpha}(z)+g_{\alpha}(z)= \displaystyle \frac{z(1-\alpha z)}{1-z^2},\,\alpha\in[-1,1],$$ be a normalized, locally univalent and sense-preserving mapping in $E$. We first prove that $f_{\alpha}$ is in $S_H$ and convex in the direction of the imaginary axis.
Since \begin{equation}
\Re[(1-z^2)F'_\alpha(z)]=\Re{\left[\frac{1+z^2-2\alpha z}{(1-z^2)}\right]}=\frac{(1-|z|^2)(1+|z|^2-2\alpha \Re{(z)})}{|1-z^2|^2}>0\,\, {\rm for\,\,all}\,\, z\in E,\end{equation} therefore, in view of Lemma 1.1, the analytic function $F_\alpha=h_{\alpha}+g_{\alpha}$ is univalent in $E$ and convex in the direction of the imaginary axis. Consequently, by Lemma 1.2, the harmonic function $f_{\alpha}=h_{\alpha}+\overline{g}_{\alpha}$ is in $S_H$ and also convex in the direction of the imaginary axis. However the harmonic mappings $f_{\alpha}=h_{\alpha}+\overline{g}_{\alpha},\,\alpha\in[-1,1],$ may not be convex in the direction of the real axis, in general (e.g. take $\alpha=-0.5$ and the dilatation $\displaystyle \omega(z)={g'(z)}/{h'(z)}=-z^2$).

In the following result we show that for the linear combination of $f_{\alpha_{1}}$ and $f_{\alpha_{2}}$ to be in $S_H$ and convex in the direction of the imaginary axis it is sufficient that the linear combination is locally univalent and sense-preserving.
\begin{theorem} Let $\displaystyle f_{\alpha_{i}}=h_{\alpha_{i}}+\overline{g}_{\alpha_{i}} \in S_H$, where $h_{\alpha_{i}}(z)+g_{\alpha_{i}}(z)=\displaystyle {z(1-\alpha_{i} z)}/{(1-z^2)},\, \alpha_{i} \in [-1,1]$ for $i=1,2$. Then the mapping $f=t f_{\alpha_{1}}+(1-t)f_{\alpha_{2}}\,,0\leq t\leq1,$ is in $S_H$ and is convex in the direction of the imaginary axis, provided $f$ is locally univalent and sense-preserving. \end{theorem}
\begin{proof} Define $F=t F_{\alpha_{1}}+ (1-t) F_{\alpha_{2}}$ where $F_{\alpha{i}}=h_{\alpha_{i}}+g_{\alpha_{i}}\,\,(i=1,2)$. Using (4), we immediately get  $$\Re{\left[(1-z^2)F'(z)\right]}=t\Re{\left[(1-z^2)F'_{\alpha_{1}}(z)\right]}+(1-t)\Re{\left[(1-z^2)F'_{\alpha_{2}}(z)\right]}>0,\,\, {\rm for\,\,all}\,\,z\in E.$$
Thus $F$ is analytic univalent and convex in the direction of the imaginary axis, by Lemma 1.1. Therefore if $f=h+\overline{g},$ where $h+g=F,$ is locally univalent and sense-preserving, then, in view of Lemma 1.2, $f \in S_H$ and maps $E$ onto a domain convex in the direction of the imaginary axis. \end{proof}
 We know that $f=h+\overline{g}$ will be locally univalent and sense-preserving if and only if $h' \not=0$ in $E$ and its dilatation function $\omega$ (say) satisfies $|\omega|<1,$ in $E$. So, we first find expression for $\omega$.

\begin{lemma} Let $\displaystyle f_{\alpha_{i}}=h_{\alpha_{i}}+\overline{g}_{\alpha_{i}}$ be in $S_H$ such that $h_{\alpha_{i}}(z)+g_{\alpha_{i}}(z)=\displaystyle {z(1-\alpha_{i} z)}/{(1-z^2)},\,\, \alpha_{i} \in [-1,1]$ for $i=1,2.$ If $\displaystyle \omega_{i}={g_i'}/{h_i'},\,i=1,2,$ are dilatation functions of $f_{\alpha_{i}},\, i=1,2,$ respectively, then the dilatation function $\omega$ of $f=t f_{\alpha_{1}}+(1-t)f_{\alpha_{2}},\,0\leq t\leq1,$ is given by
\begin{equation}\displaystyle \omega(z)= \left[\frac{(1+z^2)(t\omega_1+(1-t)\omega_2+\omega_1\omega_2)-2z(\alpha_{1} t \omega_1+\alpha_{1} t \omega_1\omega_2+(1-t)\omega_2\alpha_{2}+(1-t)\omega_1\omega_2\alpha_{2})}
{(1+z^2)(1+t\omega_2+(1-t)\omega_1)-2z(\alpha_{2}+\alpha_{1} t \omega_2+(1-t)\alpha_{2}\omega_1+\alpha_{1} t-\alpha_{2} t)}\right].\end{equation}\end{lemma}
\begin{proof} As $f=t f_{\alpha_{1}}+(1-t)f_{\alpha_{2}}=t h_1+(1-t)h_2+t\overline{g}_1+(1-t)\overline{g}_{2}$\,\, so, $$\displaystyle \omega=\frac{t g_1'+(1-t)g_2'}{t h_1'+(1-t) h_2'}=\frac{t \omega_1 h_1'+(1-t)\omega_2 h_2'}{t h_1'+(1-t) h_2'}.$$ From $h_{\alpha_{i}}(z)+g_{\alpha_{i}}(z)=\displaystyle \frac{z(1-\alpha_{i} z)}{1-z^2}$  and $\displaystyle \omega_{i}=\frac{g_i'}{h_i'},\,i=1,2,$ we get $$h'_1(z)=\frac{1+z^2-2\alpha_{1} z}{(1+\omega_1)(1-z^2)^2}\quad{\rm and}\quad h'_2(z)=\frac{1+z^2-2\alpha_{2} z}{(1+\omega_2)(1-z^2)^2}.$$
 Thus, substituting these into $\omega$ gives
$$ \begin{array}{clll}
\displaystyle \omega(z)=\displaystyle \frac{t\omega_1(1+z^2-2\alpha_{1} z)(1+\omega_2)+(1-t)\omega_2(1+z^2-2\alpha_{2} z)(1+\omega_1)}{t(1+\omega_2)(1+z^2-2\alpha_{1} z)+(1-t)(1+\omega_1)(1+z^2-2\alpha_{2} z)},
\end{array} $$
which reduces to (5) after rearrangement of terms in the numerator and denominator.\end{proof}
\begin{theorem} Let $\displaystyle f_{\alpha_{i}}=h_{\alpha_{i}}+\overline{g}_{\alpha_{i}} \in S_H$, where $h_{\alpha_{i}}(z)+g_{\alpha_{i}}(z)=\displaystyle {z(1-\alpha_i z)}/{(1-z^2)},\,\alpha_i \in [-1,1]$ for $i=1,2.$ If\, $\alpha_1=\alpha_2,$ then $f=t f_{\alpha_{1}}+(1-t)f_{\alpha_{2}},$ \,$0\leq t\leq1 $, is in $S_H$ and is convex in the direction of the imaginary axis.\end{theorem}
\begin{proof} In view of Theorem 2.1, it suffices to show that $f$ is locally univalent and sense-preserving. If $\omega_1$, $\omega_2$ and $\omega$ are dilatations of $f_{\alpha_{1}}$, $f_{\alpha_{2}}$ and $f$ respectively, then by setting  $\alpha_{1}=\alpha_{2}$ in (5), we get $$\displaystyle \omega=\frac{t\omega_1+(1-t)\omega_2+\omega_1\omega_2}{1+t\omega_2+(1-t)\omega_1}.$$
From the proof of Theorem 3 in [\ref{wa and li}], we get $|\omega|<1$. Hence $f$ is locally univalent and sense-preserving.
\end{proof}
 By taking  $\alpha_{1}=\alpha_{2}=-1$ in Theorem 2.3, we get the following result.
 \begin{corollary} If $f_i=h_i+\overline{g}_{i}\in S_H$ with $\displaystyle h_i(z)+g_i(z)={z}/{(1-z)}$ for $i=1,2$, then $f=tf_1+(1-t)f_2,\,0\leq t\leq 1,$ is in $S_H$ and is convex in the direction of the imaginary axis.\end{corollary}
  Michalski [\ref{mi}], defined the class $COD_H(\theta)$ consisting of functions $f\in S_H$, which map the unit disk $E$ onto domains convex in directions of the lines $z=te^{i \theta}\,,t\in\mathbb{R}$ and $z=te^{i(\theta+\frac{\pi}{2})}\,,t\in\mathbb{R}$ for each $\theta\in[0,{\pi}/{2}).$ Combining results of Theorem 1.4 and Corollary 2.4, we immediately get the following result.
 \begin{theorem}  Let $f_i=h_i+\overline{g}_{i} \in S_H$ where $\displaystyle h_i(z)+g_i(z)={z}/{(1-z)}\,\,for \,\,i=1,2$. Then, $f=tf_1+(1-t)f_2,\,0\leq t\leq 1,$ is in $COD_H(0).$\end{theorem}
The following lemma, popularly known as Cohn's Rule, will be required in proving our next result.
\begin{lemma} ([\ref{ra and sc}, p.375])  Given a polynomial $p(z)= a_0 + a_1z + a_2z^2+...+a_nz^n$ of
degree $n$, let $$ p^*(z)=\displaystyle z^n\overline{p\left(\frac{1}{\overline z}\right)} = \overline {a}_n + \overline {a}_{n-1}z + \overline {a}_{n-2}z^2 +...+ \overline{a}_0z^n.$$
Denote by $r$ and $s$ the number of zeros of $p$ inside and on the unit circle $|z|=1$, respectively.
 If $|a_0|<|a_n|,$ then $$ p_1(z)= \frac{\overline {a}_np(z)-a_0p^*(z)}{z}$$ is of degree $n-1$ and has $r_1=r-1$ and $s_1=s$ number
 of zeros inside and on the unit circle $|z|=1$, respectively.\end{lemma}
 We now prove the following.
\begin{theorem} Let $f_{\alpha_{i}}=h_{\alpha_{i}}+\overline{g}_{\alpha_{i}}$  be in $S_H$ where $h_{\alpha_{i}}(z)+g_{\alpha_{i}}(z)=\displaystyle {z(1-\alpha_i z)}/{(1-z^2)},\,\,\alpha_i \in [-1,1],$ for $i=1,2$. If $\omega_1(z)=-z$ and  $\omega_2(z)=z$ are dilatations of $f_{\alpha_{1}}$ and $f_{\alpha_{2}}$ respectively, then $f=t f_{\alpha_1}+(1-t)f_{\alpha_2},\,\,0\leq t\leq 1,$ is in $S_H$ and is convex in the direction of the imaginary axis provided $\alpha_1\geq\alpha_2.$\end{theorem}
\begin{proof} In view of Theorem 2.1, it is sufficient to show that dilatation $\omega$ of $f$ satisfies $|\omega|<1$ in $E$. By using the shearing technique, we explicitly get $h_i$ and $g_i,\,\,i=1,2,$ as follows:
$$\indent\hspace{-2.8cm}h_1(z)=\frac{(1-\alpha_1)}{4(1-z)^2}-\frac{(1+\alpha_1)}{4(1+z)}+\frac{(1+\alpha_1)}{8}\log\left[\frac{1+z}{1-z}\right]+\frac{\alpha_1}{2},$$
$$g_1(z)=\frac{z(1-\alpha_1 z)}{1-z^2}-\frac{(1-\alpha_1)}{4(1-z)^2}+\frac{(1+\alpha_1)}{4(1+z)}-\frac{(1+\alpha_1)}{8}\log\left[\frac{1+z}{1-z}\right]-\frac{\alpha_1}{2};$$ and
$$\indent\hspace{-2.8cm}h_2(z)=\frac{(1-\alpha_2)}{4(1-z)}-\frac{(1+\alpha_2)}{4(1+z)^2}+\frac{(1-\alpha_2)}{8}\log\left[\frac{1+z}{1-z}\right]+\frac{\alpha_2}{2},$$
$$g_2(z)=\frac{z(1-\alpha_2 z)}{1-z^2}-\frac{(1-\alpha_2)}{4(1-z)}+\frac{(1+\alpha_2)}{4(1+z)^2}-\frac{(1-\alpha_2)}{8}\log\left[\frac{1+z}{1-z}\right]-\frac{\alpha_2}{2}.$$

The case when $\alpha_1=\alpha_2$ follows from Theorem 2.3. So, we shall only consider the case when $\alpha_1>\alpha_2$. Setting $\omega_1(z)=-z$ and $\omega_2(z)=z$ in (5) we get $$\omega(z)=\displaystyle \left[\frac{(1+z^2)(-tz+(1-t)z-z^2)-2z(-\alpha_1 t z-\alpha_1 t z^2+(1-t)\alpha_2 z-(1-t)\alpha_2 z^2)}
{(1+z^2)(1+t z-(1-t)z)-2z(\alpha_2+\alpha_1 t z-(1-t)\alpha_2 z+\alpha_1 t-\alpha_2 t)}\right]$$
\begin{equation}=-z\frac{[z^3+(2t-1-2\alpha_1 t-2\alpha_2(1-t))z^2+(1+2\alpha_2(1-t)-2\alpha_1 t)z+(2t-1)]}{[(2t-1)z^3+(1+2\alpha_2(1-t)-2\alpha_1 t)z^2+(2t-1-2\alpha_1 t-2\alpha_2(1-t))z+1]}.\end{equation}
 Let \\$\indent\hspace{.2cm}\displaystyle \gamma(z)=z^3+(2t-1-2\alpha_1 t-2\alpha_2(1-t))z^2+(1+2\alpha_2(1-t)-2\alpha_1 t)z+(2t-1)$\\
 $\indent\hspace{.9cm}= a_3z^3+a_2z^2+a_1z+a_0$ \\and\\ $\indent\hspace{.2cm}\displaystyle\gamma^*(z)=(2t-1)z^3+(1+2\alpha_2(1-t)-2\alpha_1 t)z^2+(2t-1-2\alpha_1 t-2\alpha_2(1-t))z+1=z^3\overline{\gamma\left(\frac{1}{\overline{z}}\right)}$ and notice that by (6), $\omega(z)=-z\displaystyle \frac{\gamma(z)}{\gamma^*(z)}.$ \\ Thus if $z_0,$ $z_0\not=0$, is a zero of $\gamma$ then $\displaystyle{1}/{\overline{z_0}}$ is a zero of $\gamma^*$. Therefore, we can write $$\displaystyle\omega(z)=-z\frac{(z+A)(z+B)(z+C)}{(1+\overline {A} z)(1+\overline {B} z)(1+\overline {C} z)}.$$ For $ |\beta|\leq1$ the function $\displaystyle \phi(z)=\frac{z+\beta}{1+\overline{\beta}z}$ maps $\bar E=\{z: |z|\leq1\},$ onto $\bar E$. So, to prove that $|\omega_3|<1$ in $E$, it suffices to show that $|A|\leq 1$, $|B|\leq 1$ and $|C|\leq1$.
 We take $t\in (0,{1}/{2})\cup({1}/{2},1) ,$ as the cases when $t=0$ or $t=1$ are trivial and the case when $t={1}/{2}$ will be dealt separately. As $|a_0|=2t-1<1=|a_3|,$ therefore, by applying Cohn's rule on $\gamma$, it is sufficient to show that all zeros of $\gamma_1$ lie inside or on $|z|=1,$ where,\\ $\indent\hspace{.2cm}\displaystyle \gamma_1(z)= \frac{a_3\gamma(z)-a_0\gamma^*(z)}{z}$
\begin{equation} \indent\hspace{-5.6cm}=4t(1-t)\left[z^2-(\alpha_1+\alpha_2)z-(\alpha_1-\alpha_2-1)\right]
 \end{equation}
 $\indent\hspace{1cm}=b_2z^2+b_1z+b_0$. \\Now, if $\alpha_1=1$ and $\alpha_2=-1$, then both the zeros of $\gamma_1$ lie on the circle $|z|=1$ and otherwise, if $\alpha_1-\alpha_2> 0,$ we have $|b_0|<|b_2|$ because $\alpha_1,\alpha_2\in [-1,1]$ and $4t(1-t)\not=0$ for $t\in (0,{1}/{2})\cup({1}/{2},1)$.  Again, by applying Cohn's rule on $\gamma_1,$ we need to show that all zeros of $\gamma_2$ lie inside or on $|z|=1,$ where\\
 $\displaystyle \gamma_2(z)= \frac{b_2\gamma_1(z)-b_0\gamma^*_1(z)}{z}\\
 \indent\hspace{.9cm}=(4t(1-t))^2(\alpha_1-\alpha_2)[(2-\alpha_1+\alpha_2)z-(\alpha_1+\alpha_2)]$ and $\gamma^*_1(z)=\displaystyle z^2\overline{\gamma_1\left(\frac{1}{\overline{z}}\right)}.$\\
  If $z_2$ is the zero of  $\gamma_2$ then $|z_2|\leq 1$ is equivalent to $(1-\alpha_1)(1+\alpha_2)\geq0$ which is true as $|\alpha_i|\leq1$ for $i=1,2.$ Hence zeros of $\gamma_1\,$ and $\gamma$ both lie in or on the unit circle $|z|=1$. \\ In case $t={1}/{2}$ we observe that$$ \gamma(z)=z[z^2-(\alpha_1+\alpha_2)z-(\alpha_1-\alpha_2-1)].$$ In view of (7) we can easily verify that all the zeros of $\gamma$ lie in or on the unit circle $|z|=1$. Hence the result is proved.
  \end{proof}
Images of $E$ under $f_{\alpha_{1}}$, $f_{\alpha_{2}}$ and $f$ are shown in Figure 1, Figure 2 and Figure 3, respectively.

\begin{figure}[ht]
  \begin{minipage}[b]{0.55\textwidth}
    \centering
    \includegraphics[width=.55\textwidth]{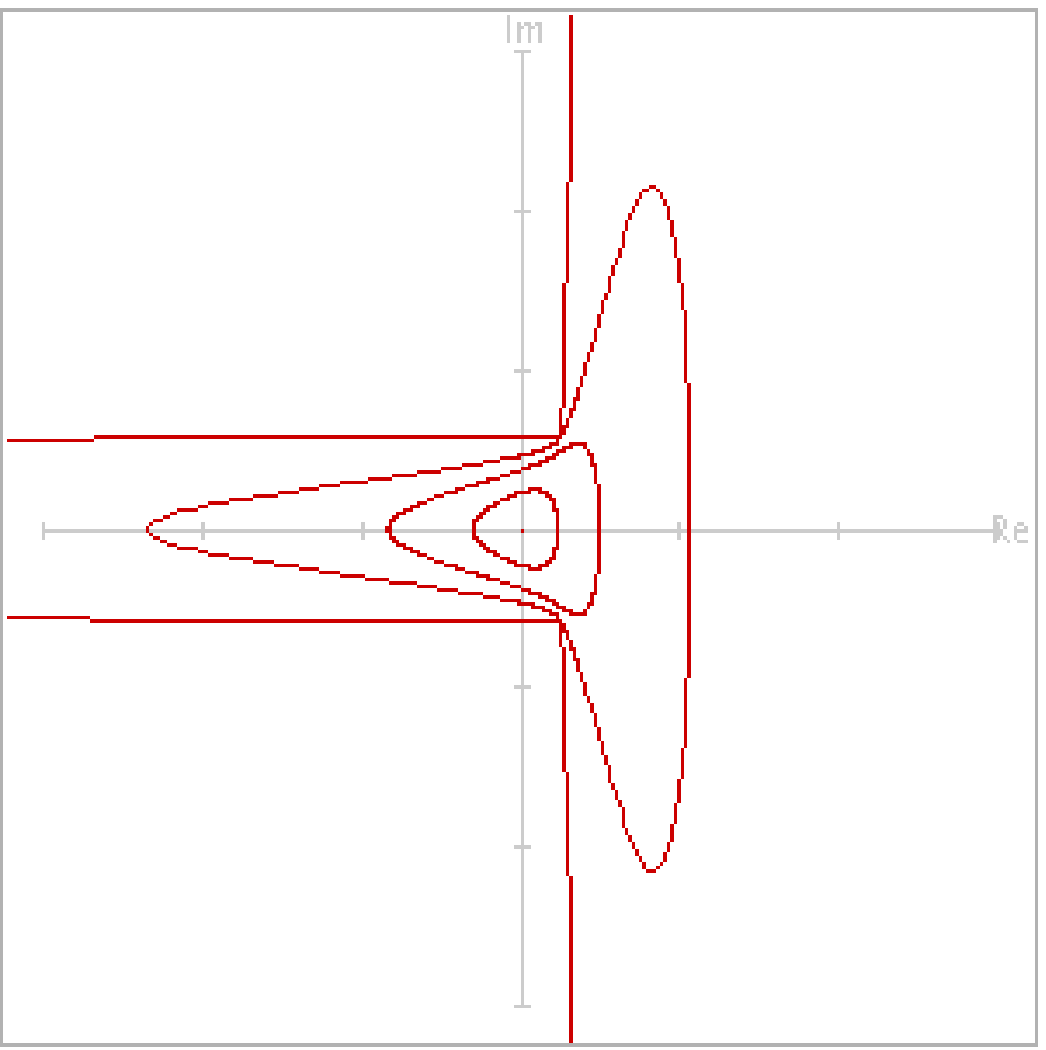}
    \vspace{-0.2cm}\caption{Image of $E$ under $f_{\alpha_{1}}$ for $\alpha_1=0.5$}
  \end{minipage}
\begin{minipage}[b]{0.55\textwidth}
\centering
\includegraphics[width=.55\textwidth]{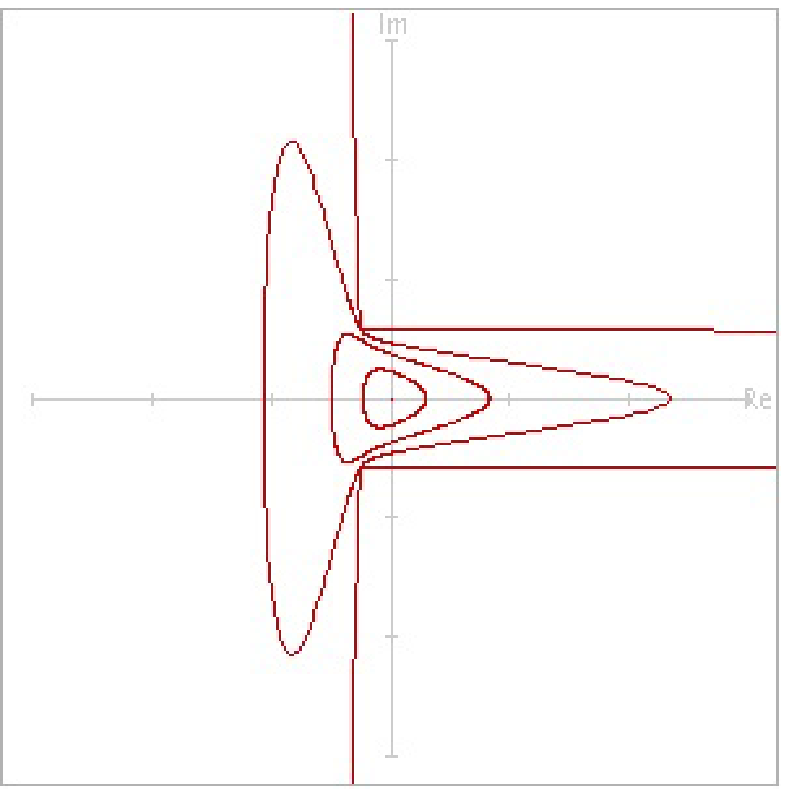}
\caption{Image of $E$ under $f_{\alpha_{2}}$ for $\alpha_2=-0.5$}
\end{minipage}
\end{figure}
\begin{figure}[ht]
\centering
\includegraphics[width=0.40\textwidth]{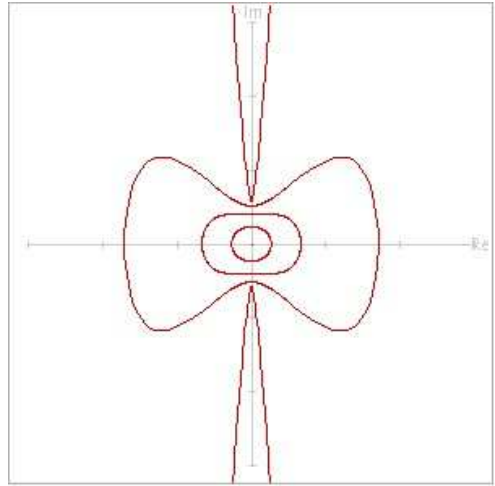}
 \caption{Image of $E$ under $f$, for $\alpha_1=0.5$, $\alpha_2=-0.5$ and $t=\frac{1}{4}$.}
\end{figure}
\begin{remark} Note that in Theorem 2.7 it is not possible to take $\alpha_2>\alpha_1$ because in that case it will then follow from (7) that the  modulus of the product of zeros of $\gamma_1$ is $|1+\alpha_2-\alpha_1|$ which is strictly greater than 1. Hence at least one zero of $\gamma_1$ and therefore of $\gamma$ shall lie outside $|z|=1$ implying that there will exist some $z\in E$ for which $|\omega(z)|\not<1$ i.e, linear combination of $f_{\alpha_{1}}$ and $f_{\alpha_{2}}$ shall no longer remain locally univalent and sense-preserving.
\end{remark}

\begin{theorem} Let $f_{\alpha_{1}}$ be same as in Theorem 2.7 and $f_{\alpha_{2}}=h_{\alpha_{2}}+\overline{g}_{\alpha_{2}} \in S_H,$ with dilatation $\omega_2$, where $h_{\alpha_{2}}(z)+g_{\alpha_{2}}(z)=\displaystyle {z(1-\alpha_2 z)}/{(1-z^2)},\,\alpha_2 \in [-1,1].$ Let $f=t f_{\alpha_{1}}+(1-t)f_{\alpha_{2}},\,0\leq t\leq 1.$ Then we have the following:\\
(i) If $\omega_2(z)=-z^2$ and $\alpha_1\geq\alpha_2$, then $f$ is in $S_H$ and is convex in the direction of the imaginary axis. \\
(ii) If $\omega_2(z)=z^2$,  $|\alpha_1|\geq|\alpha_2|$ and  $\alpha_1\alpha_2\geq0,$ then $f$ is in $S_H$ and is convex in the direction of the imaginary axis.
\end{theorem}
 As the proof runs on the same lines as that of Theorem 2.7, hence is omitted.

\begin{remark} If we take $\omega_2(z)=z^3$ in the above theorem, then we observe that $f$ may not be locally univalent and sense-preserving. For $t={3}/{4}$ if we set $\alpha_1=0.4$ and $\alpha_2=0.3$ or $\alpha_1=0.3$ and $\alpha_2=0.6$ , it can be easily verified that $|\omega_3|\not<1$ in $E$.  \end{remark}

\begin{remark} We observe that proceeding on the same lines as in Theorem 2 of Wang et al. [\ref{wa and  li}], one can easily get the following:\\
  Let $\displaystyle f_{\alpha_{i}}=h_{\alpha_{i}}+\overline{g}_{\alpha_{i}}$ be in $S_H$ such that $h_{\alpha_{i}}(z)+g_{\alpha_{i}}(z)=\displaystyle {z(1-\alpha_{i} z)}/{(1-z^2)},\,\alpha_{i} \in [-1,1]$ for $i=1,2$ and let $\displaystyle \omega_{i}={g_i'}/{h_i'},\,i=1,2,$ be dilatation functions of $f_{\alpha_{i}},\, i=1,2,$ respectively. Then $f=t f_{\alpha_{1}}+(1-t)f_{\alpha_{2}},$ $0\leq t\leq1 $, is in $S_H$ and convex in the direction of the imaginary axis if $\Re\left((1-\omega_1\overline{\omega_2})h_1'\overline{h_2'}\right)>0.$ \end{remark}

We close this paper by considering one of the harmonic functions involved in the linear combination obtained by shearing of analytic strip mapping (3).

  \begin{theorem} Let $f_\theta=h_\theta+\overline{g}_\theta\in S_H$ where $\displaystyle  h_\theta(z)+g_\theta(z)= \frac{1}{2i\,\sin\theta} \log\left(\frac{1+ze^{i\theta}}{1+ze^{-i\theta}}\right),$ $\theta\in(0,\pi)$  and $f_\alpha=h_\alpha+\overline{g}_\alpha\in S_H$ be the map such that $h_\alpha(z)+g_\alpha(z)=\displaystyle {z(1-\alpha z)}/{(1-z^2)},\,\,\alpha \in [-1,1].$ Then  $f_{\theta,\alpha}=tf_\theta+(1-t)f_\alpha,\,0\leq t\leq 1,$ is in $S_H$ and is convex in the direction of the imaginary axis provided $f_{\theta,\alpha}$ is locally univalent and sense-preserving.  \end{theorem}
\begin{proof} In view of Theorem 2.1 and (4), we need only to show that $\Re\left[(1-z^2)F_\theta'(z)\right]>0,$ where $F_\theta=h_\theta+g_\theta$. Let $$\displaystyle \phi(z)=(1-z^2)F_\theta'(z)=\frac{1-z^2}{(1+ze^{i\theta})(1+ze^{-i\theta})}.$$ Since $\phi(0)=1$ and for each $\gamma\in \mathbb{R}$, $\Re[\phi(e^{i\gamma})]=0$, therefore, by minimum principle for harmonic functions, we have, $\Re[\phi(z)]=\Re\left[(1-z^2)F_\theta'(z)\right]>0,$ for  $z\in E.$ Hence we have our result.\end{proof}
 The following example illustrates the result of above theorem.
\begin{example} Let $f_\theta=h_1+\overline{g}_1 \in S_H$ be the harmonic map considered in Theorem 2.12 with $\theta={\pi}/{2}$ and $\displaystyle \omega_1(z)={g_1'(z)}/{h_1'(z)}=-z.$ Take $f_\alpha=h_\alpha+\overline{g}_\alpha \in S_H$ such that $\displaystyle h_\alpha(z)+g_\alpha(z)={z}/{(1-z)}$ and $\displaystyle \omega_2(z)={g_2'(z)}/{h_2'(z)}=z^2.$ By shearing we get,\\ $\displaystyle h_1(z)=\frac{1}{2}\tan^{-1}z-\frac{1}{2}\log(1-z)+\frac{1}{4}\log(1+z^2),$\qquad
$\displaystyle g_1(z)=\frac{1}{2}\tan^{-1}z+\frac{1}{2}\log(1-z)-\frac{1}{4}\log(1+z^2);$ \\and \\$\displaystyle h_2(z)=\frac{z}{2(1-z)}-\frac{1}{2}\log(1-z)+\frac{1}{4}\log(1+z^2)$,\qquad $\displaystyle g_2(z)=\frac{z}{2(1-z)}+\frac{1}{2}\log(1-z)-\frac{1}{4}\log(1+z^2)$. \\

Now if $\omega$ is the dilatation of $f_{\theta,\alpha}=t f_\theta+(1-t)f_\alpha,\,0\leq t\leq1,$ then, $$\displaystyle \left|\omega\right|=\left|\frac{t g_1'+(1-t)g_2'}{t h_1'+(1-t) h_2'}\right|=\left|\frac{t \omega_1 h_1'+(1-t)\omega_2 h_2'}{t h_1'+(1-t) h_2'}\right|=\left|\frac{z(z-t)}{1-t z}\right|<1.$$ This implies that $f_{\theta,\alpha}$ is locally univalent and sense-preserving in $E$. So, in view of Theorem 2.12, $f_{\theta,\alpha} \in S_H$ and is convex in the direction of the imaginary axis. \\ Images of $E$ under $f_\theta$, $f_\alpha$ and $f_{\theta,\alpha}$ are shown in Figure 4, Figure 5 and Figure 6, respectively.
\end{example}
\begin{figure}[ht]
\begin{minipage}[b]{0.5\textwidth}
    \centering
    \includegraphics[width=.5\textwidth]{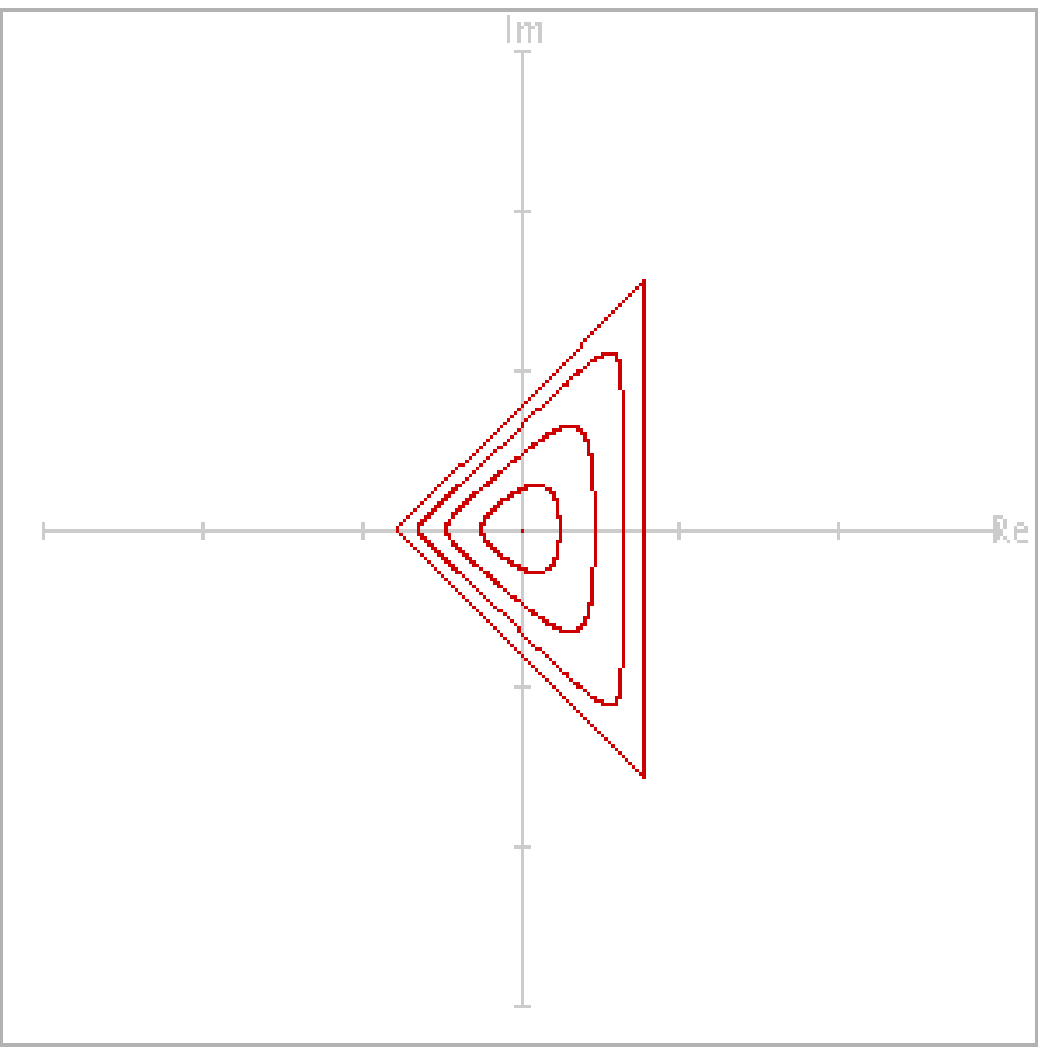}
    \caption{Image of $E$ under $f_\theta$ for $\theta=\frac{\pi}{2}$}
  \end{minipage}
  \hspace{.2cm}
  \begin{minipage}[b]{0.5\textwidth}
    \centering
    \includegraphics[width=.5\textwidth]{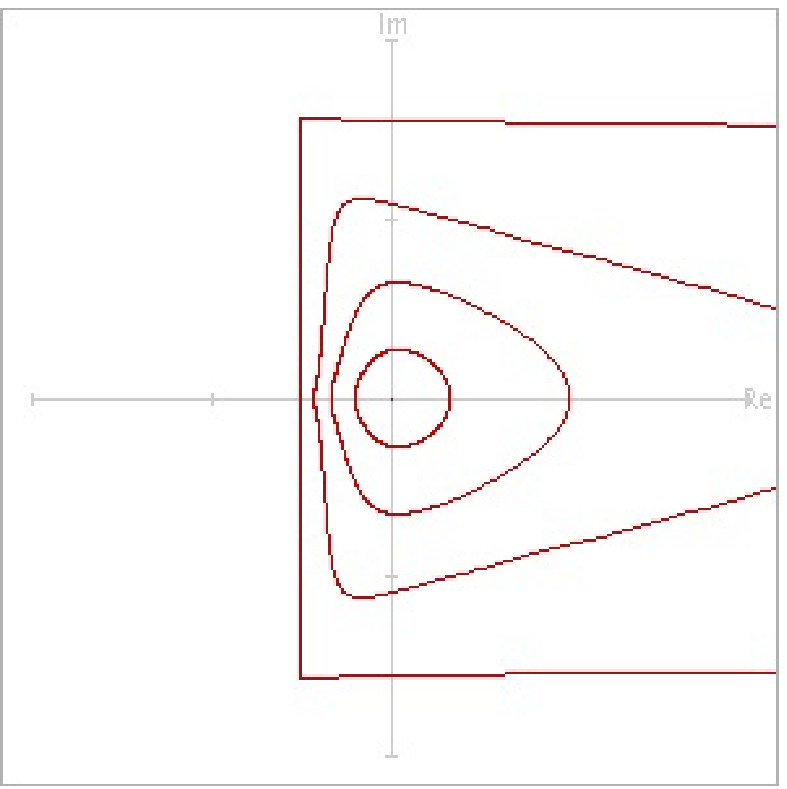}
    \caption{Image of $E$ under $f_\alpha$ for $\alpha=-1$}
  \end{minipage}
\end{figure}
\begin{figure}[ht]
\centering
\includegraphics[width=0.35\textwidth]{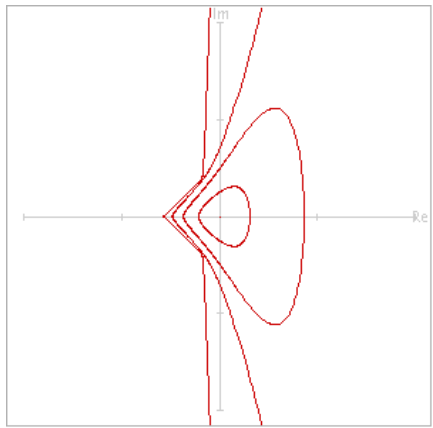}
\vspace{-0.2cm} \caption{Image of $E$ under $f_{\theta,\alpha}$ for $\theta=\frac{\pi}{2}$, $\alpha=-1$ and $t=\frac{3}{4}.$}
\end{figure}

\noindent{\emph{Acknowledgement: First author is thankful to the Council of Scientific and Industrial Research, New Delhi, for financial support vide grant no. 09/797/0006/2010 EMR-1.}}
{

\vspace{1cm}
{\footnotesize
\noindent{Department of Mathematics, \\Sant Longowal Institute of Engineering and Technology, \\Longowal-148106 (Punjab), India.}}

\end{document}